\documentclass[12pt]{amsart}
\usepackage{amsmath}
\usepackage{amssymb}
\usepackage[all]{xy}

\newtheorem{theorem}{Theorem}  
\newtheorem{lemma}[theorem]{Lemma}
\newtheorem{proposition}[theorem]{Proposition}
\newtheorem{example}[theorem]{Example}
\newtheorem{corollary}[theorem]{Corollary}
\newtheorem{remar}[theorem]{Remark}
\renewenvironment{proof}{Proof:\ \ \ }{\QED}

\newcommand{\QED}{\qed}
\newcommand{\bfind}[1]{\index{#1}{\bf #1}}
\newcommand{\n}{\par\noindent}

\newcommand{\sn}{\par\smallskip\noindent}
\newcommand{\mn}{\par\medskip\noindent}
\newcommand{\bn}{\par\bigskip\noindent}
\newcommand{\pars}{\par\smallskip}
\newcommand{\parm}{\par\medskip}

\newcommand{\R}{\mathbb R}

\newcommand{\N}{\mathbb N}

\newcommand{\cA}{\mathcal A}
\newcommand{\cB}{\mathcal B}

\newcommand{\cS}{\mathcal S}
\newcommand{\cN}{\mathcal N}
\newcommand{\cP}{\mathcal P}

\newcommand{\fun}{\mbox{\rm f-un}}

\renewcommand{\int}{\mbox{\rm int}}

\newcommand{\cal}{\mathcal}

\newcommand{\oc}[1]{{\rm \bf \,|\,#1\,|{}}}
\newcommand{\cc}{{\cal C}}
\newcommand{\mcc}[3]{{ [#1,#2]_{{\rm\bf  #3}}}}

\newcommand{\bdm}{\begin{displaymath}}
\newcommand{\edm}{\end{displaymath}}

\newcommand{\timestring}{\begingroup
   \count0 = \time
   \divide\count0 by 60
   \count2 = \count0   
   \count4 = \time
   \multiply\count0 by 60
   \advance\count4 by -\count0   
   \ifnum\count4<10
      \toks1 = {0}%
   \else
      \toks1 = {}%
   \fi
   \ifnum\count2<10
      \toks0 = {0}%
   \else
      \toks0 = {}%
   \fi
   \ifnum\count2=0
      \toks0 = {00}
   \fi
  \the\toks0 \number\count2:\the\toks1 \number\count4 \thinspace
\endgroup}%

\begin{document}
\title{Construction of ball spaces and the notion of continuity}
\author[Bartsch, Kuhlmann and Kuhlmann]{Ren\'e Bartsch, Katarzyna Kuhlmann and Franz-Viktor Kuhlmann}
\address{A.-v.-Humboldt-Gymnasium, Makarenkostr. 54, 17491 Greifswald, Germany}
\email{math@marvinius.net}
\address{Institute of Mathematics, ul.~Wielkopolska 15, 70-451 Szczecin, Poland}
\email{Katarzyna.Kuhlmann@usz.edu.pl}
\address{Institute of Mathematics, ul.~Wielkopolska 15, 70-451 Szczecin, Poland}
\email{fvk@usz.edu.pl}

\date{October 8, 2018}
%

\begin{abstract}
Spherically complete ball spaces provide a framework for the proof of generic fixed point theorems. For the
purpose of their application it is important to have methods for the construction of new spherically complete
ball spaces from given ones. Given various ball spaces on the same underlying set, we discuss the construction of
new ball spaces through set theoretic operations on the balls. A definition of continuity for functions on ball
spaces leads to the notion of quotient spaces. Further, we show the existence of products and coproducts and use
this to derive a topological category associated with ball spaces.
\end{abstract}
\maketitle


%
%
\section{Introduction}
Ball spaces have been introduced in \cite{[KK1]} as a framework for the proof of generic fixed point theorems for
functions which in some way are contracting. Since then, the development of their theory and their applications has led to several articles (\cite{[KKu],[KK2],[KK3],[KKP],[KKSh],[KKSo]}).

A \bfind{ball space} is a pair $(X,\cB)$ consisting of a nonempty set
$X$ and a nonempty set of distinguished nonempty subsets $B$ of $X$.
If we denote the power set of $X$ by $\cP(X)$, this means that
\[
\emptyset\>\ne\>\cB \>\subseteq \>\cP(X)\setminus\{\emptyset\}\>.
\]
The elements $B$ of $\cB$ will be called \bfind{balls}, in analogy to the case of metric (or ultrametric) balls.
In analogy to the case of ultrametric spaces, we will call a nonempty collection $\cN$ of balls in
$\cB$ a \bfind{nest of balls (in $\cB$)} if it is totally ordered by inclusion. We will say that $(X,\cB)$
is \bfind{spherically complete} if the intersection $\bigcap\cN$ of each nest $\cN$ of balls in $\cB$ is nonempty.

\pars
Ball spaces can be derived from various settings, such as metric spaces, ultrametric spaces, topological spaces
(where it is convenient to take the balls to be the nonempty closed sets), partially ordered sets, lattices.
The generic fixed point
theorems that hold in spherically complete ball spaces then allow specializations to all of these settings,
in which spherically completeness is equivalent to natural completeness or compactness properties
(cf.\ \cite{[KK3]}). In
addition, the general framework allows transfers of concepts and approaches between these various settings
(again, cf.\ \cite{[KK3]}).

\pars
There are many different procedures to construct a new ball space $(X,\cB)$ from a given set of ball spaces
$(X_i,\cB_i)$, $i\in I$. We will discuss several of them in this paper. Since most of the generic fixed point
theorems work for spherically complete ball spaces, it is an important question under which conditions the spaces
$(X_i,\cB_i)$ being spherically complete implies that so is $(X,\cB)$. An easy example for this to work is when we
take the union of the sets of balls of two ball spaces on the same set $X$ (see Proposition~\ref{B1cupB2}).
In certain cases we may need stronger forms of spherical completeness, which we will introduce in
Section~\ref{secthier}. This for instance happens when we are given a single ball space $(X,\cB)$ and want to
replace $\cB$ by the set of all finite unions of balls in $\cB$ (see Theorem~\ref{f-un}). For the study of further
operations on the set of balls and the connection with topologies, see \cite{[KK3]}.

Returning to the case of several ball spaces on a given set $X$, we will discuss in Section~\ref{sectbb} the
particularly interesting natural example of ordered abelian groups and fields. For the case of fields,
S.~Shelah introduced in \cite{[S]} the notion of \bfind{symmetrical completeness} which means that for every
Dedekind cut, the cofinality of the left cut set is different from the coinitiality of the right cut set (as is
the case in the reals). He showed that arbitrarily large symmetrically complete ordered fields exist. With a
different construction idea, the existence result is reproven in \cite{[KKSh]} and generalized to the case of
ordered abelian groups and ordered sets. For the ordered sets, the existence result in fact follows already
from the comprehensive work of Hausdorff. It is also shown in \cite{[KKSh]} that
symmetrical completeness is equivalent to the spherical completeness of the induced order ball space, which is
defined as follows. If $(I,<)$ is any nonempty totally ordered set, then we take $\cB_o$ to consist of all
closed bounded intervals $[a,b]$ with $a,b\in I$, $a\leq b$. We call $(I,\cB_o)$ the \bfind{order ball space}
associated with $(I,<)$.

On the other hand, non-archimedean ordered abelian groups and fields carry a nontrivial natural valuation (cf.\
Section~\ref{sectnatval}), which induces an ultrametric ball space (cf.\ Section~\ref{sectbultval}). This ball
space is always spherically complete when the order ball space is, but the two spaces are distinct. It should be
noted that the balls in the ultrametric ball space are precisely all cosets of the principal convex subgroups;
here a convex subgroup is called principal if it is the smallest among all convex subgroups containing a fixed
element.

We can apply
Proposition~\ref{B1cupB2} to these two ball spaces, but we cannot apply Theorem~\ref{f-un} as the order ball space
and hence also the union of the two ball spaces does not satisfy the necessary requirements. However, in
Section~\ref{sectbb} we prove:
\begin{theorem}                 \label{thmbb}
Take a symmetrically complete ordered group or field $G$ and $\cB$ to be the set
of all convex sets in $G$ that are finite unions of closed bounded
intervals and ultrametric balls. Then $(G,\cB)$ is spherically complete.
\end{theorem}
\n
We also show that the convex sets appearing in the theorem can always be represented as two distinct ultrametric balls connected by a closed interval (hence we gave them the nickname ``bar-bells'').

\sn
{\bf Open question:} \ {\it Does the theorem also hold if the condition ``convex'' is removed?}

\parm
Take two ball spaces $(X,\cB)$ and $(X',\cB')$ and a function $f:X\rightarrow X'\,$. The question arises under
which conditions $f$ transfers spherical completeness from one side to the other. In analogy to the case of
topological spaces, we will call $f$ \bfind{ball continuous} if the preimage of every ball in $\cB'$ is a
ball in $\cB$. Then the following holds:
\sn
{\it If $f$ is ball continuous and $(X,\cB)$ is spherically complete, then so is $(X',\cB')$.}
\sn
This result is part of Theorem~\ref{sphctransf} in Section~\ref{sectcon}, in which we study further conditions for
the transfer of spherical completeness. It provides the necessary background for the definition of the
notion of ``quotient ball space'' in the same section.

\pars
Having the notion of ball continuity at hand, we define in Section~\ref{sectprodco} the category of ball spaces,
where the morphisms are the ball continuous functions. This leads to further constructions of new ball spaces from given ones: products and coproducts. We will show that products and coproducts of spherically complete ball spaces
are again spherically complete (see Theorem~\ref{sphcprod}).

\parm
The final result of our paper (see Section~\ref{topcat}) is inspired by the fact that in many situations the
behaviour of ball spaces reminds us of
topological spaces, and that indeed many inspiring examples come from a topological context. Hence in order to
provide topological methods, we derive from the category of ball spaces the category of augmented ball spaces
by adding the full space and the empty set to the sets of balls. We prove:
\begin{theorem}                            \label{abstopcat}
The category of augmented ball spaces is topological.
\end{theorem}
\n
For definitions and details, see Section~\ref{topcat}. Note that this result implies that the category can be
embedded in a cartesian closed topological category, providing natural function spaces; see e.g.\
\cite{[Schwarz]}.

%
%
\section{Preliminaries}                \label{sectprel}
%
%
\subsection{A hierarchy of ball spaces}      \label{secthier}
In \cite{[KK3]} we introduce and study the following hierarchy of spherical completeness properties:
\sn
{\bf S}$_1$: The intersection of each nest in $(X,\cB)$ is nonempty.
\sn
{\bf S}$_2$: The intersection of each nest in $(X,\cB)$ contains a ball.
\sn
{\bf S}$_3$: The intersection of each nest in $(X,\cB)$ contains a
largest ball.
\sn
{\bf S}$_4$: The intersection of each nest in $(X,\cB)$ is a ball.
\sn
{\bf S}$_i^c$: The same as {\bf S}$_i$, but with ``centered system'' in place of ``nest''.
\sn
Here, a \bfind{centered system of balls} is a
collection of balls such that the intersection of any finite number of
balls in the collection is nonempty.

\bn
\subsection{Ball spaces associated with ultrametric spaces and valuations} \label{sectbultval}
An \bfind{ultrametric space} is a pair $(X,u)$ where $X$ is a set and $u:X\times X\to\Gamma$ with
$\Gamma$ a totally ordered set with largest element $\infty$, such that:

\sn
{\bf (UM1)} $u(x,y)=\infty$ if and only if $x=y$;\n
{\bf (UM2)} $u(x,y)=u(y,x)$;\n
{\bf (UM3)} $u(x,z)\ge\min\{u(x,y),u(y,z)\}$.
\sn
(UM3) is called the \bfind{ultrametric triangle law}. The image
\[
uX:=u(X\times X)\setminus\{\infty\}
\]
is called the \bfind{value set}
of $u$. A generalization of the notion of ultrametric space works with partially ordered value sets $uX$ in place
of totally ordered ones (see \cite{[PR1],[PR2],[KK3],[KKu]}), but we will not need this generalization in the
present paper.

\pars
One frequent source of ultrametrics are valuations: if $v$ is a valuation on a field or an
abelian group, then one may define $u(a,b):=v(a-b)$. With this definition, the above axioms are satisfied if $v$
is written as a Krull valuation, that is, it satisfies the following axioms:
\sn
{\bf (V1)} $vx=\infty$ if and only if $x=0$;\n
{\bf (V2)} $v(x-y)\ge\min\{vx,vy\}$ \quad (ultrametric triangle law).
\sn
Note that for valuations on fields we require in addition that $\Gamma$ is an ordered abelian group
together with $\infty$ (an element larger than all elements of the group) and the following axiom holds:
\sn
{\bf (VH)} $v(xy)=vx+vy$ \quad (homomorphism law).
\sn

\pars
Take an ultrametric space $(X,u)$. For $x\in X$ and $\gamma\in uX\cup\{\infty\}$, the set
\[
B_\gamma(x)\>:=\>\{x\in X\mid u(x,y)\ge\gamma\}
\]
is called the \bfind{closed ball of radius $\gamma$ around $x$}. Further, we define
\[
B(x,y)\>:=\>B_{u(x,y)}(x)\>=\>B_{u(x,y)}(y)\>,
\]
where the latter equality follows from the fact that in an ultrametric space every element of a ball is its center.
If $X$ is nonempty, then $(X,\cB_u)$ with
\[
\cB_u\>:=\>\{B(x,y)\mid x,y\in X\}
\]
is a ball space, which we call the \bfind{ultrametric ball space}.

From the ultrametric triangle law it follows that any two balls in $\cB_u$ are already comparable by inclusion once
they have a nonempty intersection. It follows that every centered system of balls in $\cB_u$ is in fact a nest of
balls.

\bn
\subsection{Nonarchimedean orderings and the natural valuation} \label{sectnatval}
Take an ordered abelian group $(G,<)$. Two elements $a,b\in G$ are
called \bfind{archime\-dean equivalent} if there is some $n\in\N$ such
that $n|a|\geq |b|$ and $n|b|\geq |a|$. The ordered abelian group
$(G,<)$ is archimedean ordered if all nonzero elements are archimedean
equivalent. If $0\leq a<b$ and $na<b$ for all $n\in\N$, then we say that
``$a$ is infinitesimally smaller than $b$''.
We denote by $va$ the archimedean equivalence class of $a$. The set of
archimedean equivalence classes can be ordered by setting $va>vb$ if and
only if $|a|<|b|$ and $a$ and $b$ are not archimedean equivalent, that
is, if $n|a|<|b|$ for all $n\in\N$. We write $\infty:=v0\,$; this is the
maximal element in the linearly ordered set of equivalence classes. The
function $a\mapsto va$ is a group valuation on $G$, i.e., it satisfies
(V1) and (V2). By definition,
\[
0\leq a\leq b\>\Longrightarrow\>va\geq vb\;.
\]
The set $vG:=\{vg\mid 0\ne g\in G\}$ is called the \bfind{value set} of
the valued abelian group $(G,v)$.

\pars
If $(K,<)$ is an ordered field, then we consider the natural
valuation on its ordered additive group and define $va+vb:=v(ab)$.
This turns the set of archimedean classes into an ordered abelian group,
with neutral element $0:=v1$ and inverses $-va=v(a^{-1})\,$. In this
way, $v$ becomes a field valuation.

\pars
As shown in Section~\ref{sectbultval} above, the natural valuation of an ordered
abelian group or ordered field induces an ultrametric ball space $\cB_u\,$.

\bn
%
%
\section{Hybrid ball spaces}                \label{secthy}
In this section we investigate the following question:
\sn
{\it Given two spherically complete ball spaces $(X,\cB_1)$ and $(X,\cB_2)$ on the same set $X$,
which operations of forming new balls from the balls in $\cB_1\cup\cB_2$
will preserve spherical completeness?}
\sn

A first step is provided by the following proposition; the easy proof is left to the reader.
\begin{proposition}                                  \label{B1cupB2}
If $(X,\cB_1)$ and $(X,\cB_2)$ are {\bf S}$_1$ ball spaces, then so is the ball space $(X,\cB_1\cup\cB_2)$. The
same holds with {\bf S}$_2$ or {\bf S}$_4$ in place of {\bf S}$_1\,$.
\end{proposition}

Note that the assertion may become false if we replace {\bf S}$_1$ by
{\bf S}$_3\,$. Indeed, the intersection of a nest in $\cB_1$ may
properly contain a largest ball which does not remain the largest ball
contained in the intersection in $\cB_1\cup\cB_2\,$.

\pars
Having obtained $\cB=\cB_1\cup\cB_2\,$, the next question is how to create new balls from the balls in $\cB$
without losing spherical completeness. The results of taking unions and intersections are discussed in
\cite{[KK3]}. In the next section, we present a particular case.

\mn
%
%
\subsection{Closure under finite unions of balls}       \label{sectun}
\mbox{ }\sn
Take a ball space $(X,\cB)$. By $\fun(\cB)$ we denote the set of all unions of finitely
many balls in $\cB$. In \cite{[KK3]}, the following theorem is proven:
\begin{theorem}                           \label{f-un}
If $(X,\cB)$ is an {\bf S}$_1^c$ ball space, then so is $(X,\fun(\cB))$.
\end{theorem}

For the convenience of the reader, we repeat the proof here. We need a lemma.
\begin{lemma}                               \label{BfB}
If ${\cal S}$ is a maximal centered system of balls in $\fun(\cB)$ (that
is, no subset of $\fun(\cB)$ properly containing ${\cal S}$ is a centered
system), then there is a subset ${\cal S}_0$ of ${\cal S}$ which is a
centered system in $\cB$ and has the same intersection as ${\cal S}$.
\end{lemma}
\begin{proof}
It suffices to prove the following: if $B_1,\ldots,B_n\in \cB$ such that
$B_1\cup\ldots\cup B_n\in \cS$, then there is some $i\in\{1,\ldots,n\}$
such that $B_i\in\cS$.

Suppose that $B_1,\ldots,B_n\in \cB\setminus \cS$. By the maximality
of $\cS$ this implies that for each $i\in\{1,\ldots,n\}$, $\cS\cup
\{B_i\}$ is not centered. This in turn means that there is a finite
subset $\cS_i$ of $\cS$ such that $\bigcap\cS_i\cap B_i=\emptyset$. But
then $\cS_1\cup\ldots\cup\cS_n$ is a finite subset of $\cS$ such that
\[
\bigcap(\cS_1\cup\ldots\cup\cS_n)\cap(B_1\cup\ldots\cup B_n)
\>=\>\emptyset\>.
\]
This yields that $B_1\cup\ldots\cup B_n\notin \cS$, which proves our
assertion.
\end{proof}

\mn
{\bf Proof of Theorem~\ref{f-un}:} \
Take a centered system $\cS'$ of balls in $\fun(\cB)$. Centered
systems of balls in a ball space are inductively ordered by
inclusion. Hence there is a maximal centered system $\cS$ of balls
in $\fun(\cB)$ which contains $\cS'$. By Lemma~\ref{BfB} there is a
centered system $\cS_0$ of balls in $\cB$ such that $\bigcap\cS_0=
\bigcap\cS\subseteq\bigcap\cS'$. Since $(X,\cB)$ is an S$_1^c$ ball
space, we have that $\bigcap\cS_0\ne\emptyset$, which yields that
$\bigcap\cS'\ne\emptyset$. This proves that $(X,\fun(\cB))$
is an S$_1^c$ ball space.
\qed

\pars
This theorem becomes false if ``{\bf S}$_1^c$'' is replaced by ``{\bf S}$_1$''.
\begin{example} \rm
For every $i\in\N$ we let $p_i$ be the $i$-th prime. For every $i\in\N$ we define a ball $B_i\subset\R$ by
\[
B_i\>:=\>\left(0,\frac{1}{p_i}\right)\setminus\left\{\frac{1}{p_i^j}\mid j\in\N \mbox{ with }
p_i^j>p_{i+1}\right\}\>,
\]
and set $\cB:=\{B_i\mid i\in\N\}$. Then for $i\ne j$, $B_i$ and $B_j$ are not comparable by inclusion. Therefore,
$\cB$ admits no nests with more than one ball and is thus spherically complete. But
\[
B_i\cup B_{i+1}\>=\> \left(0,\frac{1}{p_i}\right)
\]
since all the real numbers in $(0,\frac{1}{p_i})$ that are missing in $B_i$ are elements of $B_{i+1}\,$. Further,
$(B_i\cup B_{i+1})_{i\in\N}$ is a nest in $(\R,\fun(\cB))$. As it has empty intersection, $(\R,\fun(\cB))$ is
not spherically complete.
\end{example}

This leads us to the following question:
\sn
{\it Under which other conditions than {\bf S}$_1^c$ is spherical completeness preserved under taking finite unions?}

\sn
We discuss one special case in the next section, starting with a ball space that is not {\bf S}$_1^c$.

\mn
%
%
\subsection{Bar-bells}                \label{sectbb}
As already mentioned in the Introduction, a natural example of algebraic structures on which two distinct
ball spaces appear naturally are ordered
abelian groups and ordered fields. On the one hand such a structure $(G,<)$ admits a natural valuation which is
nontrivial if the ordering is nonarchimedean (cf.\ Section~\ref{sectnatval}). This gives rise to an ultrametric
space, from which in turn we can derive the ball space $(G,\cB_u)$ where $\cB_u$ consists of all closed
ultrametric balls (cf.\ Section~\ref{sectbultval}). On the other hand, one can consider the order
ball space $(G,\cB_o)$ where $\cB_o$ consists of all nonempty closed bounded intervals. In \cite{[KKSh]} it is
shown that if $(G,<)$ is symmetrically complete, then $(G,\cB_o)$ is spherically complete,
and that if $(G,\cB_o)$ is spherically complete, then so is $(G,\cB_u)$. Hence if $(G,<)$ is symmetrically
complete, then $(G,\cB)$ is spherically complete for $\cB=\cB_u\cup\cB_o$ by Proposition~\ref{B1cupB2}.

While $(G,\cB_u)$ is always an {\bf S}$_2^c$ ball space once it is spherically complete (cf.\ \cite{[KK3]}),
$(G,\cB_o)$ and $(G,\cB)$ are {\bf S}$_1^c$ if and only if $G=\R$ (with the canonical ordering). Hence if
$G\ne\R$, we cannot apply Theorem~\ref{f-un} here.

\parm
A \bfind{bar-bell} is a subset of $G$ obtained from a nonempty closed bounded interval $[a,b]$ by joining it with
ultrametric balls centered in $a$ and in $b$; it can thus be written as $B_\alpha(a)\cup [a,b] \cup B_\beta(b)$
with $\alpha,\beta\in uG$ and $a\leq b$.

\begin{lemma}                 \label{barbellssphc}
For every symmetrically complete ordered abelian group or field, the ball space of bar-bells is spherically complete.
\end{lemma}
\begin{proof}
Take a nest $\cN$ of bar-bells. W.l.o.g.\ we may assume that $\cN=(B_i)_{i<\kappa}$ where $\kappa$ is the
cofinality of the nest and that $i<j<\kappa\Rightarrow B_j\subsetneq B_i\,$. We write
\[
B_i\>=\> B_{\alpha_i}(a_i)\cup [a_i,b_i] \cup B_{\beta_i}(b_i)
\]
with $\alpha_i,\beta_i\in uG$ and $a_i\leq b_i\,$.

If there is a nest of ultrametric balls that has an intersection which is a
subset of the intersection of $\cN$, then we are done. Hence assume that there is no such nest. We will construct a
sequence $(i_\nu)_{\nu<\kappa}$ that is cofinal in $\kappa$, such that
\begin{equation}                              \label{BB'}
B_{i_{\nu+1}}\>\subsetneq\> [a_{i_{\nu}},b_{i_{\nu}}] \>\subseteq\> B_{i_\nu}\>.
\end{equation}
Then
\[
\bigcap_{i<\kappa} B_i\>=\> \bigcap_{\nu<\kappa} B_{i_\nu}\>=\> \bigcap_{\nu<\kappa} [a_{i_{\nu}},b_{i_{\nu}}]
\]
and we are done again.

We take $i_0=0$.  Assume that for some $\nu<\kappa$ we have already chosen $i_\mu$ for all $\mu\leq\nu$ such
that (\ref{BB'}) holds with $\mu$ in place of $\nu$. By our
assumption there must be some $i_{\nu+1}<\kappa$, $i_{\nu+1}>i_\nu$, such that $B_{\alpha_{i_{\nu+1}}}(a_{i_{\nu+1}})
\not\subset B_{\alpha_{i_\nu}}(a_{i_\nu})$ and $B_{\beta_{i_{\nu+1}}}(b_{i_{\nu+1}})\not\subset
B_{\beta_{i_\nu}}(b_{i_\nu})$. Then $B_{\alpha_{i_{\nu+1}}}(a_{i_{\nu+1}})\cap B_{\alpha_{i_\nu}}(a_{i_\nu})=
\emptyset$ and $B_{\beta_{i_{\nu+1}}}(b_{i_{\nu+1}})\cap B_{\beta_{i_\nu}}(b_{i_\nu})=\emptyset$. Since
$B_{i_{\nu+1}}\subset B_{i_\nu}\,$, we must have that $a_{i_\nu}<B_{i_{\nu+1}}<b_{i_\nu}$.
Therefore, (\ref{BB'}) holds.

Now assume that $\lambda<\kappa$ is a limit ordinal and that we have already chosen $i_\nu$ and
constructed $B'_{i_\nu}$ for all $\nu<\lambda$ such that (\ref{BB'}) holds. Then we choose any $i\in I$ that is
larger than all $i_\nu$ (which exists since $\lambda$ is smaller than the cofinality $\kappa$) and set
$i_\lambda:=i$ and proceed as above with $\nu=\lambda$.
\end{proof}
\pars
Any finite union $S$ of closed bounded intervals and ultrametric balls that is convex is a bar-bell.
This is seen as follows. Suppose that the union of the intervals $[a_i,b_i]$, $1\leq i\leq m$, and the balls
$B_{\alpha_j}(c_j)$, $1\leq j\leq n$, is convex. Since ultrametric balls with equal centers are comparable by
inclusion, by listing only the largest ones we may assume that all $c_j$ are distinct. Further, we can add the
singleton ball $B_\infty(a_i)$ or $B_\infty(b_i)$ to the balls $B_{\alpha_j}(c_j)$ in case $a_i\,$, or $b_i$
respectively, is not contained in any of the balls. Let $\min_{1\leq j\leq n} c_j=c_{j_1}$ and
$\max_{1\leq j\leq n} c_j=c_{j_2}\,$. Then $S=B_{\alpha_{j_1}}(c_{j_1})\cup [c_{j_1},c_{j_2}]\cup
B_{\alpha_{j_2}}(c_{j_2})$.

\pars
From this fact together with Lemma~\ref{barbellssphc} we obtain Theorem~\ref{thmbb}.


\parm
Let us collect a few special properties of the ball space $\cB$ of convex finite unions of closed bounded
intervals and ultrametric balls, which could be helpful in answering the question stated after
Theorem~\ref{thmbb}.
\sn
1) All of its balls can be expressed by a union of at most 3 balls from the two generating ball spaces $\cB_u$ and
$\cB_o\,$.
\sn
2) Every ball in $\cB_u\cap\cB_o$ is a singleton.
\sn
3) Both $\cB_u$ and $\cB_o$ are closed under finite nonempty intersections.
\sn
4) If $\cN$ is a nest in $\cB$l, then there is a nest $\cN'$ in $\cB_u$ or in $\cB_o$ such that
$\bigcap\cN'\subseteq \bigcap\cN$.

\parm
An attempt to generalize the notion of convexity to arbitrary ball spaces could be to call a finite collection of
balls \bfind{pseudo convex} if it is of the form $\{B_1,\ldots,B_n\}$ with $B_i\cap B_{i+1}\ne \emptyset$ for $1\leq
i<n$.

\sn
{\bf Open question:} \ {\it Does the closure of a ball space under unions of finite pseudo convex collections of
balls always preserve spherical completeness, even if the ball space fails to be {\bf S}$_1^c$? If not, what other
conditions could ensure this?}

\bn
%
%
\section{Ball continuity and quotient ball spaces}  \label{sectcon}
Take two ball spaces $(X,\cB)$ and $(X',\cB')$ and a function
$f:X\rightarrow X'\,$. We will call $f$
%
%
\bfind{ball continuous} if the preimage of every ball in $\cB'$ is a
ball in $\cB$, that is,
\begin{equation}                 \label{Blr}
\{f^{-1}(B')\mid B'\in \cB'\}\>\subseteq\> \cB\>.
\end{equation}
We will call $f$ \bfind{ball closed} if the image of every ball in $\cB$ is a ball in $\cB'$.
\begin{lemma}                    \label{compcont}
Take three ball spaces $(X,\cB_X)$, $(Y,\cB_Y)$ and $(Z,\cB_Z)$, and functions $f:X\rightarrow Y$ and
$g:Y\rightarrow
Z$. Then $g\circ f$ is ball continuous if $f$ and $g$ are. Likewise, $g\circ f$ is ball closed if $f$ and $g$ are.
\end{lemma}
\begin{proof}
For the first assertion, note that for every $B\in\cB_Z$, $(g\circ f)^{-1}(B)=f^{-1}(g^{-1}(B))$. The
second assertion is obvious.
\end{proof}

\pars
The next theorem gives conditions for functions to preserve spherical completeness in one or the other ditrection.
\begin{theorem}                         \label{sphctransf}
Take two ball spaces $(X,\cB)$ and $(X',\cB')$, and a function $f:X\rightarrow X'\,$.
\sn
a) If $f$ is ball continuous and $(X,\cB)$ is spherically complete, then so is $(X',\cB')$.
\sn
b) Assume that $f$ is surjective and
\begin{equation}                 \label{Brl}
\cB \>\subseteq\> \{f^{-1}(B')\mid B'\in \cB'\}\>.
\end{equation}
If $(X',\cB')$ is spherically complete, then so is $(X,\cB)$.

\sn
c) Assume that $f$ is surjective or $(X',\cB')$ is an {\bf S}$_2$ ball space, and that
\begin{equation}                 \label{B=}
\cB \>=\> \{f^{-1}(B')\mid B'\in \cB'\}\>.
\end{equation}
Then $(X,\cB)$ is spherically complete if and only if $(X',\cB')$ is.

\sn
d) If (\ref{B=}) holds and $f$ is surjective, then
\begin{equation}                 \label{B'=}
\cB' \>=\> \{f(B)\mid B\in \cB\}
\end{equation}
and $f$ induces an isomorphism of the partially ordered sets $\cB$ and $\cB'$.

\sn
e) $f:X\rightarrow X'$ is ball closed and finite-to-one, and if $(X',\cB')$ is spherically complete, then so is
$(X,\cB)$.
%
\end{theorem}
\begin{proof}
a): \ Take a nest $\cN'=(B'_i)_{i\in I}$ of balls in $(X',\cB')$. Set $\cN=(f^{-1}(B'_i))_{i\in I}\,$. By assumption,
we have that $f^{-1}(B'_i)\in\cB$ for all $i\in I$, For any $i,j
\in I$ we have that $B'_i\subseteq B'_j$ or $B'_j\subseteq B'_i\,$, hence also $f^{-1}(B'_i)\subseteq f^{-1}(B'_j)$
or $f^{-1}(B'_j)\subseteq f^{-1}(B'_i)\,$. This proves that $\cN$ is a nest in $\cB$. As $(X,\cB)$ is spherically
complete,
$\bigcap\cN$ is nonempty. Since $f(\bigcap\cN)\subseteq B'_i$ for all $i\in I$, it follows that $f(\bigcap\cN)
\subseteq\bigcap\cN'$, which shows that the latter is nonempty.

\pars
b): \ Take a nest $\cN=(B_i)_{i\in I}$ of balls in $(X,\cB)$. Set $\cN'=(f(B_i))_{i\in I}\,$. By assumption,
we have that every $B_i$ is the preimage of a ball $B'_i$ in $\cB'$, hence by surjectivity of $f$,
$f(B_i)=B'_i\in\cB'$ for all $i\in I$.
For any $i,j\in I$ we have that $B_i\subseteq B_j$ or $B_j\subseteq B_i\,$, hence also $f(B_i)\subseteq f(B_j)$ or
$f(B_j)\subseteq f(B_i)\,$. This proves that $\cN'$ is a nest in $\cB'$. As $(X',\cB')$ is spherically complete,
$\bigcap\cN'$
is nonempty. Take $x'\in\bigcap\cN'$. Then $x'\in B'_i$ for all $i\in I$. Pick some preimage $x\in X$ of $x'$ under
$f$; this is possible since $f$ is assumed surjective. Then $x\in f^{-1}(B'_i)=B_i$ for all $i\in I$. Hence $x\in
\bigcap\cN$, showing that this intersection is nonempty.

\pars
c): \ If $f$ is surjective, everything follows from assertions a) and b). If $f$ is not surjective, but $(X',\cB')$
is an {\bf S}$_2$ ball space, then we have to modify the previous proof in order to show that some $x'\in\bigcap
\cN'$ has a preimage in $X$. Since $(X',\cB')$ is {\bf S}$_2\,$, $\bigcap\cN'$ contains a ball $B'$. By (\ref{B=}),
$f^{-1}(B')$ is a ball in $\cB$ and hence nonempty. So there is an $x'\in\bigcap\cN'$ which has a preimage in $X$.

\pars
d): \ A surjective function $f:X\rightarrow X'$ induces an inclusion preserving bijection between all subsets of
$X'$ and their preimages. Condition (\ref{B=}) implies (\ref{B'=}) and that the restriction of $f$ to $\cB'$ is a
bijection between $\cB'$ and $\cB$.

\pars
e): \ Take a nest $\cN=(B_i)_{i\in I}$ of balls in $(X,\cB)$. Set $\cN'=(f(B_i))_{i\in I}\,$. By assumption, each
$f(B_i)$ is a ball in $\cB'$, and since $f$ preserves inclusion between balls, $\cN'$ is a nest. As $(X',\cB')$ is
spherically complete, $\bigcap\cN'$ is nonempty. Take $x'\in\bigcap\cN'$. Then $x'\in f(B_i)$ for all $i\in I$.
Among the finitely many preimages of $x'$ under $f$ there must be at least one that is contained in all $B_i\,$.
This element then lies in $\bigcap\cN$, showing that the intersection is nonempty.
\end{proof}

\parm
Take any function $f:X\rightarrow X'$. We define functions
\[
\varphi_f:\cP(X)\rightarrow \cP(X') \mbox{ \ and \ } \psi_f:\cP(X')\rightarrow \cP(X)
\]
by setting $\varphi_f(S):=f(S)$ for each set $S\subseteq X$ and $\psi_f(S')
:=f^{-1}(S')$ for each set $S'\subseteq X'$. If $\cS\subseteq\cP(X)$ and $\cS'\subseteq\cP(X')$, then in accordance
with our general use for functions, $\varphi_f(\cS)=\{\varphi_f(S)\mid S\in\cS\}$ and $\psi_f(\cS')=
\{\psi_f(S')\mid S'\in\cS'\}$. If $(X,\cB)$ is a ball space, then $\cB$ is just a nonempty subset of $\cP(X)
\setminus\{\emptyset\}$. We observe that then $\varphi_f(\cB)$ is a nonempty subset of $\cP(X')\setminus
\{\emptyset\}$. This shows that $\varphi_f$ sends ball spaces on $X$ to ball spaces on $X'$. Similarly, if $f$
is surjective, then $\psi_f$ sends ball spaces on $X'$ to ball spaces on $X$.

Every nest of balls in $\cB$ is also a ball space on $X$, but it has the special property of being totally ordered
by inclusion. So we are interested in the question when this property is preserved by $\varphi_f$ and $\psi_f\,$.
The following is a corollary to the previous proof.

\begin{corollary}
Take two ball spaces $(X,\cB)$ and $(X',\cB')$ and a function $f:X\rightarrow X'$.
\sn
a) If $f$ is ball continuous and $\cN'$ is a nest of balls in $\cB'$, then $\psi_f(\cN')$ is a nest of balls
in $\cB$.
\sn
b) If $f$ is ball closed and $\cN$ is a nest of balls in $\cB$, then $\varphi_f(\cN)$ is a nest of balls
in $\cB'$.
\end{corollary}

Take two ball spaces $(X,\cB)$ and $(X',\cB')$. If there is a surjective function $f:X\rightarrow X'$ such that
(\ref{B=}) holds, then we call $(X',\cB')$ a  \bfind{quotient ball space} of $(X,\cB)$ (induced by the function
$f$). Note that $\cB$ is the coarsest of all ball spaces $\cS$ on $X$ such that $f$ is a ball continuous function
from $(X,\cS)$ to $(X',\cB')$, and that $\cB'$ is the finest of all ball spaces $\cS'$ on $X'$ such that $f$ is a
ball continuous function from $(X,\cB)$ to $(X',\cS')$.


%
%
\section{Products and coproducts}               \label{sectprodco}
The \bfind{category of ball spaces} consists of all ball spaces as objects and the ball continuous functions between
them as morphisms. In this section we show that products and coproducts exist in this category, and
we will describe them explicitly.
\begin{theorem}
The category of ball spaces admits products and coproducts.
\end{theorem}

For the proof of the theorem, we will explicitly construct these objects. Take ball spaces $(Y_i,\cB_i)$, $i\in I$.
\sn
{\bf 1) Products.} \ We set $X=\prod_{i\in I} Y_i$ and denote by $p_i:X\rightarrow Y_i$
the projection from $X$ to $Y_i\,$. We set
\begin{equation}                         \label{Bprod}
\cB\>:=\> \left\{\prod_{i\in I} B_i\subseteq X\mid \mbox{ for some }k\in I,\> B_k\in\cB_k\mbox{ and }
\forall i\ne k: B_i=Y_i \right\}\>.
\end{equation}
Since the sets $\cB_i$ are nonempty, it follows that $\cB\ne\emptyset$, and as no ball in any $\cB_i$ is empty, it follows that no ball in $\cB$ is empty.

The definition of $\cB$ yields that
all projections are ball continuous. We have to show that for any ball space $(Z,\cB_Z)$ and ball continuous
functions $f_i:Z\rightarrow Y_i$ there is a unique ball continuous function $g:Z\rightarrow X$ such that
$p_i\circ g=f_i$ for all $i\in I$. The latter forces $g(z)=(f_i(z))_{i\in I}$ for all $z\in Z$; this ensures
uniqueness. It remains to show that $g$ is ball continuous. A ball $B\in\cB$ is of the form $\prod_{j\in I} S_j$
as in (\ref{Bprod}). Then $g^{-1}(B)=f_k^{-1}(B_k)$, which is a ball in $\cB_Z$ since $f_k$ is ball continuous.
This shows that $g$ is ball continuous.

\mn
{\bf 2) Coproducts.} \
Take ball spaces $(Y_i,\cB_i)$, let $X$ be the disjoint union $\dot\bigcup_{i\in I} Y_i$ and denote by
$\iota_i:Y_i\rightarrow X$ the canonical embedding of $Y_i$ in $X$. We set
\begin{equation}                      \label{Bcoprod}
\cB\>:=\> \left\{\bigcup_{i\in I} \iota_i(B_i) \mid \forall i\in I: B_i\in\cB_i\right\}\>.
\end{equation}
For the same reasons as before, $\cB\ne\emptyset$ and no ball in $\cB$ is empty.

For all $j\in I$ we have that $\iota_j^{-1}(\bigcup_{i\in I} \iota_i(B_i))=B_j\,$, so each $\iota_j$ is ball
continuous. We have to show that for any ball space $(Z,\cB_Z)$ and ball continuous
functions $f_i:Y_i\rightarrow Z$ there is a unique ball continuous function $g:X\rightarrow Z$ such that
$g\circ\iota_i=f_i$ for all $i\in I$. The latter forces $g(x)=f_i(y)$ when $y\in Y_i$ with $x=\iota_i(y)$; this
ensures uniqueness. It remains to show that $g$ is ball continuous. Take a ball $B\in\cB_Z$. Then $g^{-1}(B)=
\bigcup_{i\in I} \iota_i(f_i^{-1}(B))$. This is a ball in $\cB$ because $f_i^{-1}(B)\in\cB_i$ for all $i\in I$
as all $f_i$ are ball continuous.

\mn
{\bf Notation:} We will denote the product defined above by $\prod_{i\in I} (Y_i,\cB_i)$, and the coproduct by
$\coprod_{i\in I} (Y_i,\cB_i)$. Further, we may rewrite $\bigcup_{i\in I} \iota_i(B_i)$ as $\dot\bigcup_{i\in I}
B_i\,$.

\begin{theorem}                                   \label{bprtpr}
Take ball spaces $(Y_i,\cB_i)$, $i\in I$. Then the following assertions hold:
\sn
a) \ $\prod_{i\in I} (Y_i,\cB_i)$ is spherically complete if and only if all ball spaces $(Y_i,\cB_i)$, $i\in I$,
are spherically complete.
\sn
b) \ If at least one of the ball spaces
$(Y_i,\cB_i)$, $i\in I$, is spherically complete, then $\coprod_{i\in I} (Y_i,\cB_i)$ is spherically complete.
\sn
c) \ The following are equivalent:
\sn
i) all ball spaces $(Y_i,\cB_i)$, $i\in I$, are {\bf S}$_2\,$,
\sn
ii) $\prod_{i\in I} (Y_i,\cB_i)$ is {\bf S}$_2\,$,
\sn
ii) $\coprod_{i\in I} (Y_i,\cB_i)$ is {\bf S}$_2\,$.
\sn
The same holds with {\bf S}$_3$ and with {\bf S}$_4$ in place of {\bf S}$_2\,$.
\end{theorem}
\begin{proof}
Let us first observe the following. If $I$ and $J$ are some index sets and $B_{i,j}$, $i\in I$, $j\in J$ are
arbitrary sets, then
\begin{equation}
\bigcap_{j\in J}^{} \prod_{i\in I} B_{i,j}\>=\> \prod_{i\in I} \bigcap_{j\in J}^{} B_{i,j}\quad\mbox{ and }\quad
\bigcap_{j\in J}^{} \dot\bigcup_{i\in I} B_{i,j}\>=\> \dot\bigcup_{i\in I} \bigcap_{j\in J}^{} B_{i,j}\>
\end{equation}
a): \ First assume that all ball spaces $(Y_i,\cB_i)$, $i\in I$, are spherically complete. Take a nest
$(\prod_{i\in I} B_{i,j})_{j\in J}$ in $\prod_{i\in I} (Y_i,\cB_i)$. Then there is some $k\in I$ such that
$B_{i,j}=Y_i$ for all $i\in I$ with $i\ne k$ and all $j\in J$, and $(B_{k,j})_{j\in J}$ is a nest of balls
in $\cB_k\,$. Since $(Y_k,\cB_k)$ is spherically complete, $\bigcap_{j\in J} B_{k,j}\ne\emptyset$. As
$\bigcap_{j\in J} B_{i,j}=Y_i$ for $i\ne k$, we find that $\bigcap_{j\in J} \prod_{i\in I} B_{i,j} = \prod_{i\in I}
\bigcap_{j\in J}^{} B_{i,j}\ne\emptyset$. This proves that $\prod_{i\in I} (Y_i,\cB_i)$ is spherically complete.

\pars
Now assume that $\prod_{i\in I} (Y_i,\cB_i)$ is spherically complete. We have already seen that
all projections are ball continuous. Hence it follows from part a) of Theorem~\ref{sphctransf} that for
every $i\in I$, $(Y_i,\cB_i)$ is spherically complete.

\sn
b): \ Assume that at least one of the ball spaces $(Y_i,\cB_i)$, $i\in I$, is spherically complete.
Take a nest $(\dot\bigcup_{i\in I} B_{i,j})_{j\in J}$ in $\coprod_{i\in I} (Y_i,\cB_i)$. Since the unions are
disjoint, $\dot\bigcup_{i\in I} B_{i,j_1}\subseteq\dot\bigcup_{i\in I} B_{i,j_2}$ implies that
$B_{i,j_1}\subseteq B_{i,j_2}$ for all $i\in I$. This shows that for each $i\in I$, $(B_{i,j})_{j\in J}$ is a
nest in $\cB_i\,$. If $(Y_k,\cB_k)$ is spherically complete, then
\[
\emptyset\>\ne\> \bigcap_{j\in J}^{} B_{k,j} \>\subseteq\> \dot\bigcup_{i\in I} \bigcap_{j\in J}^{} B_{i,j}
\>=\> \bigcap_{j\in J}^{} \dot\bigcup_{i\in I} B_{i,j}\>.
\]

c): \ We keep the notations of the proofs of a) and b). If all ball spaces $(Y_i,\cB_i)$, $i\in I$, are
{\bf S}$_2\,$, then for each $i\in I$, $\bigcap_{j\in J}^{} B_{i,j}$ contains a ball $B'_i$ and hence
$\bigcap_{j\in J}^{} \prod_{i\in I} B_{i,j}=
\prod_{i\in I} \bigcap_{j\in J}^{} B_{i,j}$ contains the ball $\prod_{i\in I} B'_i$. If all spaces are even
{\bf S}$_3\,$, then the $B'_i$ can be taken as maximal balls in $\cB_i\cup\{Y_i\}$ contained in $\bigcap_{j\in J}^{}
B_{i,j}$ and it follows that $\prod_{i\in I} B'_i$ is a maximal ball contained in $\prod_{i\in I}
\bigcap_{j\in J}^{} B_{i,j}$. If all spaces are {\bf S}$_4\,$, then for each $i\in I$, $\bigcap_{j\in J}^{} B_{i,j}$
is a ball in $\cB_i\cup\{Y_i\}$ and hence $\bigcap_{j\in J}^{} \prod_{i\in I} B_{i,j}= \prod_{i\in I} \bigcap_{j\in
J}^{} B_{i,j}$ is a ball. This proves that i) implies ii) in all three cases.

In view of the equality $\bigcap_{j\in J}^{} \dot\bigcup_{i\in I} B_{i,j}=\dot\bigcup_{i\in I} \bigcap_{j\in J}^{}
B_{i,j}$ the above arguments can be readily adapted to prove that i) implies iii) in all three cases.

\pars
Now assume that $\prod_{i\in I} (Y_i,\cB_i)$ is {\bf S}$_2\,$. Take $k\in I$ and a nest of balls
${\cal N}= (B_j)_{j\in J}$ in $(Y_k,\cB_k)$. We define the nest $\prod_{i\in I} (Y_i,\cB_i)$ as in the proof of
part a). By assumption, the intersection
$\bigcap_{j\in J}^{} \prod_{i\in I} B_{i,j}= \prod_{i\in I} \bigcap_{j\in J}^{} B_{i,j}$ contains a ball
$\prod_{i\in I} B'_i$. It follows that $B'_k$ is a ball in $\cB_k$ that is contained in
$\bigcap_{j\in J}^{} B_{k,j}$. If $\prod_{i\in I} (Y_i,\cB_i)$ is even {\bf S}$_3\,$, then $\prod_{i\in I} B'_i$
can be assumed to be a maximal ball contained in $\bigcap_{j\in J}^{} \prod_{i\in I} B_{i,j}$, which implies that
$B'_k$ is a maximal ball contained in $\bigcap_{j\in J}^{} B_{k,j}$. If $\prod_{i\in I} (Y_i,\cB_i)$ is
{\bf S}$_4\,$, then $\bigcap_{j\in J}^{} \prod_{i\in I} B_{i,j}$ is ball, which implies that $\bigcap_{j\in J}^{}
B_{k,j}$ is a ball in $\cB_k\,$. We have proved that ii) implies i) in all three cases.

Replacing ``$\prod_{i\in I}$'' by ``$\dot\bigcup_{i\in I}$'' in the proof we just gave renders a proof for the fact
that iii) implies i) in all three cases.
\end{proof}

\pars
At first glance it may be surprising that for $\coprod_{i\in I} (Y_i,\cB_i)$ to be spherically complete it
suffices that just one of the ball spaces $(Y_i,\cB_i)$ is spherically complete, while for
$\coprod_{i\in I} (Y_i,\cB_i)$ to be {\bf S}$_2\,$, {\bf S}$_3$ or {\bf S}$_4\,$, {\it all} ball spaces
$(Y_i,\cB_i)$ must have the same property. This is because for
$\bigcap_{j\in J}^{} \dot\bigcup_{i\in I} B_{i,j}$ to be nonempty it suffices that
$\bigcap_{j\in J}^{} B_{i,j}$ is nonempty for at least one $i\in I$, whereas for it to contain a ball, all
$\bigcap_{j\in J}^{} B_{i,j}$  must contain a ball.

\pars
From the previous theorem we immediately obtain the following result.
\begin{theorem}                        \label{sphcprod}
The categories of spherically complete ball spaces, of {\bf S}$_2$ ball spaces, of {\bf S}$_3$ ball spaces
and of {\bf S}$_4$ ball spaces, all of them with ball continuous functions as their morphisms, admit products
and coproducts.
\end{theorem}

Let us conclude this section by giving an example which shows that the converse in part b) of Theorem~\ref{bprtpr}
is in general not true. (However, it can be shown to be true when all $\cB_i$ are countable.)
\begin{example}
\rm
Take $X_1=\omega$, $\cB_1$ to be the set of all final segments in $\omega$, $X_2=\omega_1$, and $\cB_2$ to be the
set of all final segments in $\omega_1\,$. Then neither $(X_1,\cB_1)$ nor $(X_2,\cB_2)$ is spherically complete as
in both cases, the intersection over all final segments is empty. Note, however, that the intersection over every
countable nest of balls in $(X_2,\cB_2)$ is nonempty, and thus the same is true in the coproduct of the two ball
spaces. On the other hand, if $\cN=(\dot\bigcup_{i\in \{1,2\}} B_{i,j})_{j\in J}$ is an uncountable nest in
$\coprod_{i\in \{1,2\}} (Y_i,\cB_i)$, then we may w.l.o.g.\ assume that $J=\omega_1$ and that $j_1<j_2$ implies
$B_{1,j_2}\subseteq B_{1,j_1}$ and $B_{2,j_2}\subseteq B_{2,j_1}$. Since $\cB_1$ is countable, the balls $B_{1,j}$
must eventually become stationary, i.e., equal to one and the same ball $B_1\in\cB_1\,$, which is then contained in
the intersection of the nest $\cN$.
\end{example}

%
%
\section{The topological category of augmented ball spaces}               \label{topcat}
If $(X,\cB)$ is a ball space and $\cA=\cB\cup\{\emptyset,X\}$, or if $X$ is an arbitrary (not necessarily nonempty)
set and $\cA=\{\emptyset,X\}$, then we
call $(X,\cA)$ an \bfind{augmented ball space}. Take two ball spaces $(X,\cB)$ and $(X',\cB')$ and a ball continuous
function $f:X\rightarrow X'\,$. Since $f^{-1}(\emptyset)=\emptyset$ and $f^{-1}(X')=X$, $f$ also satisfies
\begin{equation}                 \label{Alr}
\{f^{-1}(A')\mid A'\in \cA'\}\>\subseteq\> \cA\>,
\end{equation}
where $\cA'=\cB'\cup\{\emptyset,X'\}$. Therefore we will also call $f$ a ball continuous function from $(X,\cA)$ to
$(X',\cA')$. Note that $f$ is always ball continuous when $\cA'=\{\emptyset,X\}$.

We define the \bfind{category of augmented ball spaces} to consist of all augmented ball spaces as objects, with
the ball continuous functions between them as morphisms.

A category $\cc$ is called \bfind{topological} if
\begin{enumerate}
\item For all $X\in\oc{Set}$ and all families $(f_i,(X_i,\xi_i))_{i\in I}$, indexed by a class $I$,  of
$\cc$-objects $(X_i,\xi_i)$ and functions $f_i:X\to X_i$ there exists a unique initial $\cc$-object $(X,\xi)$ on the set $X$, i.e., an object $(X,\xi)$ s.t.
for all objects $(Y,\eta)\in \oc{\cc}$ and maps $g:Y\to X$ the following holds:
\bdm
g\in \mcc{(Y,\eta)}{(X,\xi)}{\cc } \quad\Leftrightarrow\quad \forall i\in I: f_i\circ g\in \mcc{(Y,\eta)}{(X_i,\xi_i)}{\cc}
\edm
\bdm
\begin{xy}
\xymatrix{
  (Y,\eta) \ar[r]^{g}&   (X,\xi)\ar[r]^{f_i} &  (X_i,\xi_i)\\
 }
\end{xy}
\edm
That is: arbitrary initial structures exist.

\item Fibre-smallness: For all $X\in\oc{Set}$, the class of $\cc$-objects on $X$ is a set.
\item On sets with at most one element exists exactly one $\cc$-structure.
\end{enumerate}

\mn
Proof of Theorem~\ref{abstopcat}:
\n
(1) \ The category admits initial objects.
Take augmented ball spaces $(Y_i,\cA_i)$, a set $X$, and functions $f_i:X\rightarrow Y_i\,$,
$i\in I$. We set
\begin{equation}                    \label{inobB}
\cA\>:=\> \{f_i^{-1}(A_i)\mid i\in I,\> A_i\in\cA_i\}\>.
\end{equation}
Observe that $\emptyset=f_i^{-1}(\emptyset)\in\cA$ and $X=f_i^{-1}(Y_i)\in\cA$ since $(Y_i,\cA_i)$  is an augmented
ball space; hence also $(X,\cA)$ is an augmented ball space.

The definition of $\cA$ yields that
all $f_i$ as functions from $(X,\cA)$ to $(Y_i,\cA_i)$ are ball continuous. We have to show that for any ball
space $(Z,\cA_Z)$ and function $g:Z\rightarrow X$ we have that $g: (Z,\cA_Z)\rightarrow (X,\cA)$ is ball continuous
if and only if for all $i\in I$, $f_i\circ g: (Z,\cA_Z)\rightarrow (Y_i,\cA_i)$ is ball continuous.

If $g$ is ball continuous, then by Lemma~\ref{compcont} and our remark at the start if this section, so are all
$f_i\circ g\,$. For the converse, assume that
all $f_i\circ g$ are ball continuous. Take $A\in\cA$. Then by definition, there is some $i\in I$ and some $A_i\in
\cA_i$ such that $A=f_i^{-1}(A_i)$. Since $f_i\circ g$ is ball continuous, $g^{-1}(A)=g^{-1}(f_i^{-1}(A_i))=
(f_i\circ g)^{-1}(A_i)\in \cA_Z\,$. This shows that $g$ is ball continuous.

\sn
(2) \ Every ball structure on $X$ is a subset of $\cP(X)$ and so it is an element of $\cP(\cP(X))$.
\sn
(3) \ Every set with at most one element carries a unique augmented ball space structure. Indeed, if the set is
empty, then this is $(\emptyset,\{\emptyset\})$. If $X$ is a singleton, then $(X,\{X\})$ is a ball space, so by
definition, $(X,\{\emptyset,X\})$ is an augmented ball space; this is the only augmented ball space structure on
$X$.
\qed

\pars
Note that the definition (\ref{inobB}) used in the proof also yields an initial object in the category of all ball
spaces where the morphisms are assumed to be the surjective ball continuous
functions. But if one of the $f_i$'s is not surjective, it can happen that there is a
ball $B_i\in\cA_i$ such that $B_i\cap f_i(X)=\emptyset$, so that $f_i^{-1}(B_i)=\emptyset$.

\parm
The definitions of product and coproduct can be taken over from the category of ball spaces. If we use the
construction described in (\ref{Bprod}) to derive an augmented ball space $(X,\cA)$ from augmented ball spaces
$(Y_i,\cA_i)$, then $\cA$ will now contain $\emptyset$ (we take $B_i=\emptyset\in\cA_i$ for some $i\in I$) and $X$
(we take $B_i=Y_i\in\cA_i$ for all $i\in I$).

Further, if we apply the construction described in (\ref{Bcoprod}) to derive an augmented ball space $(X,\cA)$ from
augmented ball spaces $(Y_i,\cA_i)$, then $\cA$ will again contain $\emptyset$ (we take $B_i=\emptyset\in\cA_i$ for
all $i\in I$) and $X$ (we take $A_i=Y_i\in\cA_i$ for all $i\in I$).

\parm
Observe that now also
\[
\cA'\>:=\> \left\{\iota_i(A_i) \mid  i\in I,\, A_i\in\cA_i\right\}
\]
renders all embeddings $\iota_i$ ball continuous since $\emptyset\in\cA_i\,$. However, in all cases where $\cA'$
differs from the set $\cA$ obtained from (\ref{Bcoprod}), examples can be constructed that show that this is not a
coproduct.

\parm
Since the category of augmented ball spaces is topological, it must also admit final objects
(cf.\ \cite{[Preuss]}). We can present them explicitly. Take augmented ball spaces $(Y_i,\cA_i)$, a set $X$, and
functions $f_i:Y_i \rightarrow X$, $i\in I$. We set
\begin{equation}                    \label{finobB}
\cA\>:=\> \{A\subseteq X\mid \forall i\in I: f_i^{-1}(A)\in\cA_i\}\>.
\end{equation}
Then all $f_i$ as functions from $(Y_i,\cA_i)$ to $(X,\cA)$ are ball continuous. Further, $\emptyset\in\cA$ since
$f_i^{-1}(\emptyset)=\emptyset\in\cA_i$ for all $i\in I$, and $X\in\cA$ since $f_i^{-1}(X)=Y_i\in\cA_i$ for all
$i\in I$. It remains to show that for any ball space $(Z,\cA_Z)$ and function $g:X\rightarrow Z$ we have that $g:
(X,\cA)\rightarrow (Z,\cA_Z)$ is ball continuous if and only if for all $i\in I$, $g\circ f_i: (Y_i,\cA_i)
\rightarrow (Z,\cA_Z)$ is ball continuous.

If $g$ is ball continuous, then by Lemma~\ref{compcont} and our remark at the start if this section, so are all
$g\circ f_i\,$. For the converse, assume that all $g\circ f_i$ are ball continuous and take $A\in\cA_Z\,$. Then
$f_i^{-1}(g^{-1}(A))=(g\circ f_i)^{-1}(A)\in\cA_i$
for all $i\in I$ by continuity. Hence by (\ref{finobB}), $g^{-1}(A)\in\cA$, showing that $g$ is ball continuous.

\pars
We observe that this definition does not work in the category of ball spaces. For example, consider ball spaces
$(X,\{B_1\})$ and $(X,\{B_2\})$ with $B_1\ne B_2$ and $f_1$ and $f_2$ the identity function of $X$. Then there is no
subset of $X$ with preimage $B_1$ under $f_1$ and $B_2$ under $f_2\,$, so (\ref{finobB}) renders the
empty set when $\emptyset$ and the whole set are not balls.

\newcommand{\lit}[1]{\bibitem{#1}}

\end{document}